\documentclass[a4paper,12pt]{amsart}
\usepackage{amssymb}
\usepackage{cite}

\usepackage{ifthen}
\usepackage[dvips]{graphicx}
\nonstopmode \numberwithin{equation}{section}
\usepackage{amssymb}
\usepackage{ifthen}
\usepackage{graphicx}
\usepackage{amsmath}
\usepackage[T1]{fontenc} 
\usepackage[utf8]{inputenc}
\usepackage[usenames,dvipsnames]{color}
\usepackage{color}
\usepackage[english]{babel}
\usepackage{fancyhdr}
\usepackage{fancybox}
\usepackage{tikz}

\theoremstyle{plain}
\newtheorem{prop}{Proposition}

\newtheorem{conj}{Conjecture}

\theoremstyle{definition}

\newtheorem{cor}{Corollary}[section]
\newtheorem{thm}{Theorem}[section]

\newtheorem{lem}{Lemma}[section]
\newtheorem{prob}{Problem}[section]
\newtheorem{rem}{Remark}[section]

\theoremstyle{plain}

\newtheorem*{thmB}{Theorem B}
\newtheorem*{thmC}{Theorem C}
\newtheorem*{thmD}{Theorem D}
\newtheorem*{thmE}{Theorem E}

\newtheorem*{lemA}{Lemma A}
\newtheorem*{lemB}{Lemma B}

\newcounter{minutes}
\divide\time by 60
\newcounter{hours}
\multiply\time by 60
\addtocounter{minutes}{-\time}

\newcounter {own}
\def\theown {\thesection       .\arabic{own}}

\newenvironment{pf}[1][]{%
	\vskip 3mm
	\noindent
	\ifthenelse{\equal{#1}{}}%
	{{\slshape Proof. }}%
	{{\slshape #1.} }%
}%
{\qed\bigskip}

\newcounter{alphabet}

\def\be{\begin{equation}}
	\def\ee{\end{equation}}

\newcommand{\bee}{\begin{enumerate}}
	\newcommand{\eee}{\end{enumerate}}

\newcommand{\blem}{\begin{lem}}
	\newcommand{\elem}{\end{lem}}
\newcommand{\bthm}{\begin{thm}}
	\newcommand{\ethm}{\end{thm}}
\newcommand{\bcor}{\begin{cor}}
	\newcommand{\ecor}{\end{cor}}
\newcommand{\beg}{\begin{examp}}
	\newcommand{\eeg}{\end{examp}}
\newcommand{\begs}{\begin{examples}}
	\newcommand{\eegs}{\end{examples}}

\newcommand{\bdefn}{\begin{defn}}
	\newcommand{\edefn}{\end{defn}}

\newcommand{\bprob}{\begin{prob}}
	\newcommand{\eprob}{\end{prob}}
\newcommand{\bei}{\begin{itemize}}
	\newcommand{\eei}{\end{itemize}}

\newcommand{\bcon}{\begin{conj}}
	\newcommand{\econ}{\end{conj}}
\newcommand{\bcons}{\begin{conjs}}
	\newcommand{\econs}{\end{conjs}}
\newcommand{\bprop}{\begin{prop}}
	\newcommand{\eprop}{\end{prop}}
\newcommand{\br}{\begin{rem}}
	\newcommand{\er}{\end{rem}}
\newcommand{\brs}{\begin{rems}}
	\newcommand{\ers}{\end{rems}}
\newcommand{\bo}{\begin{obser}}
	\newcommand{\eo}{\end{obser}}
\newcommand{\bos}{\begin{obsers}}
	\newcommand{\eos}{\end{obsers}}
\newcommand{\bpf}{\begin{pf}}
	\newcommand{\epf}{\end{pf}}
\newcommand{\ba}{\begin{array}}
	\newcommand{\ea}{\end{array}}
\newcommand{\beq}{\begin{eqnarray}}
	\newcommand{\beqq}{\begin{eqnarray*}}
		\newcommand{\eeq}{\end{eqnarray}}
	\newcommand{\eeqq}{\end{eqnarray*}}

\begin{document}

\title{{Bohr-Type Inequalities for Shifted Disks via Optimal $H^2$-Embeddings}}

\author{Molla Basir Ahamed}
\address{Molla Basir Ahamed, Department of Mathematics, Jadavpur University, Kolkata-700032, West Bengal, India.}
\email{mbahamed.math@jadavpuruniversity.in}

\author{Vasudevarao Allu$^*$}
\address{Vasudevarao Allu, Department of Mathematics, School of Basic Science, Indian Isntitute of Technology Bhubaneswar, Bhubaneswar-752050, odisha, India.}
\email{avrao@iitbbs.ac.in}

\author{Rajesh Hossain}
\address{Rajesh Hossain, Department of Mathematics, Jadavpur University, Kolkata-700032, West Bengal, India.}
\email{rajesh1998hossain@gmail.com}

\author{Taimur Rahman}
\address{Taimur Rahman, Department of Mathematics, Jadavpur University, Kolkata-700032, West Bengal, India.}
\email{taimurr.math.rs@jadavpuruniversity.in}

\subjclass[{AMS} Subject Classification:]{{Primary 30A10, 30H05, 30C35 Secondary:30C55, 41A58.}}
\keywords{{Analytic functions, Schwarz lemma, Bohr’s inequality, unimodular bounded functions, shifted disk, half plan}}
\def\thefootnote{}
\footnotetext{ {\tiny File:~\jobname.tex,
printed: \number\year-\number\month-\number\day,
          \thehours.\ifnum\theminutes<10{0}\fi\theminutes }
} \makeatletter\def\thefootnote{\@arabic\c@footnote}\makeatother
\maketitle
\begin{abstract} 
 The primary objective of this paper is to systematically generalize this phenomenon by replacing the standard unit disk with a family of nested, internally tangent shifted disks $\Omega_\gamma$ parameterized by $\gamma \in [0, 1)$, defined by$$\Omega_\gamma = \left\{ z \in \mathbb{C} : \left| z + \frac{\gamma}{1 - \gamma} \right| < \frac{1}{1 - \gamma},\; \gamma\in [0, 1) \right\}.$$ By exploiting the geometric characteristics of $\Omega_\gamma$ and evaluating the limiting behavior as $\gamma \to 1^-$, we establish a novel framework to determine the Bohr radius for the unbounded half-plane $\mathbb{H}_1 = \{z \in \mathbb{C} : \text{Re}(z) < 1\}$. Furthermore, we prove several sharp variations of the Bohr inequality within these domains, including refined and improved formulations for unimodular bounded analytic functions. The results obtained herein not only extend classical radius problems to unbounded regions but also illuminate the delicate interplay between domain deformation and coefficient estimates.
\end{abstract}
\pagestyle{myheadings}
\markboth{M. B. Ahamed, V. Allu, R. Hossain and T. Rahaman}{Asymptotic deformation of shifted disks and optimal $H^2$-norm embeddings in Bohr-type inequalities}
\section{\bf Introduction}\label{sec-1}
Let $\mathbb{D}_r(a) = \{z \in \mathbb{C} : \vert{}z - a\vert{} < r\}$ denote the open disk centered at $a \in \mathbb{C}$ with radius $r > 0$, and set $\mathbb{D} := \mathbb{D}_1(0)$ as the open unit disk in the complex plane $\mathbb{C}$. Let $\Omega \subseteq \mathbb{C}$ be a simply connected domain such that $\mathbb{D} \subseteq \Omega$, and denote by $\mathcal{H}(\Omega)$ the space of analytic functions on $\Omega$. We consider the subclass of bounded analytic functions given by$$\mathcal{B}(\Omega) = \{f \in \mathcal{H}(\Omega) : f(\Omega) \subseteq \overline{\mathbb{D}}\}.$$For any function $f \in \mathcal{B}(\Omega)$, its localized representation in $\mathbb{D}$ is given by the Taylor series expansion $f(z) = \sum_{n=0}^{\infty} \alpha_n z^n$. The associated majorant series is defined as $M_f(r) = \sum_{n=0}^{\infty} \vert{}\alpha_n\vert{} r^n$ for $r = \vert{}z\vert{} < 1$. Following Fournier and Ruscheweyh \cite{Fournier-Ruscheweyh-CRMPLN-2010}, the Bohr radius $\mathcal{B}_\Omega$ \cite{Bohr-PLMS-1914} associated with the class $\mathcal{B}(\Omega)$ is defined by
\begin{align*}
	\mathcal{B}_\Omega = \sup \left\{ r \in (0, 1) : M_f(r) \le 1 \text{ for all } f(z) = \sum_{n=0}^{\infty} \alpha_n z^n \in \mathcal{B}(\Omega) \right\}.
\end{align*}\vspace*{2mm}

In particular, it is well-known that when $\Omega = \mathbb{D}$, $\mathcal{B}_\mathbb{D} = 1/3$. This is described as follows:

\medskip
\noindent
{\bf Theorem A.} {\it (The classical Bohr theorem) If $f \in \mathcal{B}(\mathbb{D})$, then $M_f(r) \le 1$ for $0 \le r \le 1/3$. The number $1/3$ is best possible.}

\begin{rem}
	It should be noted that for $r > 1/3$, the inequality $M_f(r) \le 1$ does not hold in general for all $f \in \mathcal{B}(\mathbb{D})$. To verify the sharpness of the radius $r = 1/3$, let us consider the conformal automorphism of the unit disk $\mathbb{D}$ defined by
	$$ \varphi_a(z) = \frac{a - z}{1 - az}, \quad a \in [0, 1). $$
	For $|az| < 1$, the function $\varphi_a$ admits the power series expansion
	$$ \varphi_a(z) = a - (1 - a^2)\sum_{n=1}^{\infty} a^{n-1} z^n. $$
	Consequently, the associated majorant function $M_{\varphi_a}(r)$ is given by
	$$ M_{\varphi_a}(r) = a + (1 - a^2)\sum_{n=1}^{\infty} a^{n-1} r^n = a + \frac{(1 - a^2)r}{1 - ar}. $$
	A straightforward calculation reveals that the inequality $M_{\varphi_a}(r) > 1$ is satisfied if, and only if,
	$$ a + \frac{(1 - a^2)r}{1 - ar} > 1, $$
	which simplifies directly to
	$$ r > \frac{1}{1 + 2a}. $$
	Taking the limit as $a \to 1^-$, the bound ${1}/{(1 + 2a)}$ approaches $1/3$. This establishes the optimality of the radius $1/3$.
\end{rem}
Historically, Bohr \cite{Bohr-PLMS-1914} originally established Theorem A for the restricted radius $r \le 1/6$. The sharp radius $r = 1/3$, widely recognized as the classical Bohr radius for the class of analytic self-maps $\mathcal{B}(\mathbb{D})$, was subsequently obtained independently by Riesz, Schur, and Wiener. Alternative proofs confirming this optimal value were later provided by Sidon \cite{Sidon-MZ-1927} and Tomic \cite{Tomic-MS-1962}. For a comprehensive survey of developments, generalizations, and recent advancements regarding the Bohr radius, {we refer to} \cite{Abu-Muhanna-Ali-Ponnusamy-2017,Ahamed-Ahammed-Hamada-AMS-2026,Ahamed-Roy-PMS-2025,Ahammed-Ahamed-CVEE-2023,Ahammed-Ahamed-CMB-2024, Alkhaleefah-Kayumov-Ponnusamy-PAMS-2019, Ahamed-Allu-RMJM-2022, Ahamed-Allu-CMB-2023} and the extensive references therein.\vspace*{2mm}

By {using} the classical Bohr inequality, Dixon \cite{Dixon-BLMS-1995} constructed a counterexample confirming the existence of a Banach algebra that satisfies a non-unital analogue of von Neumann's inequality but fails to be isomorphic to an operator algebra. This seminal work revitalized interest in the Bohr radius problem, revealing profound connections to modern functional analysis and operator theory. Subsequently, Boas and Khavinson \cite{Boas-Khavinson-PAMS-1997} systematically generalized the classical formulation to the setting of several complex variables, thereby introducing the multi-dimensional Bohr radius. Over the past few decades, this geometric phenomenon has been extensively investigated and extended across diverse structural frameworks, including abstract Banach spaces, Dirichlet series, and generalized complex domains. For a comprehensive overview of these multi-dimensional developments and their applications, we refer the to \cite{Aizenberg-PAMS-2000, Defant-Frerick-Ortega-Cerda-Ounaies-Seip-AM-2011} and the extensive references therein.\vspace*{2mm}

For $\gamma\in [0, 1)$, we consider the open disk $\Omega_\gamma$ which is now known as shifted disks defined by
\begin{align*}
	\Omega_\gamma = \left\{ z \in \mathbb{C} : \left| z + \frac{\gamma}{1 - \gamma} \right| < \frac{1}{1 - \gamma},\; \gamma\in [0, 1) \right\}.
\end{align*}
Because the disks expand entirely to the left while sharing a common rightmost boundary point at $z = 1$, they exhibit a strict nested inclusion behavior. {In particular,} 
	\begin{align*}
		\text{if } 0 \le \gamma_1 < \gamma_2 < 1, \quad \text{then } \Omega_{\gamma_1} \subsetneq \Omega_{\gamma_2}.
	\end{align*} For $\gamma = 0$, we get the standard open unit disk  $\Omega_0 = \{z \in \mathbb{C} : |z| < 1\} = \mathbb{D}$. Since $\Omega_0 \subseteq \Omega_\gamma$ for all $\gamma \in [0, 1)$, every shifted disk in this family contains the unit disk {$\mathbb{D}$}. Thus, the boundary circles $\partial\Omega_\gamma$ converge to the vertical line $\text{Re}(z) = 1$. Since the interior of each disk lies to the left of this boundary, the union of all disks as $\gamma \to 1^-$ completely fills the open half-plane 
	\begin{align*}
		\lim_{\gamma \to 1^-} \Omega_\gamma = \{ z \in \mathbb{C} : \text{Re}(z) < 1 \}:=\mathbb{H}_1.
\end{align*}
\begin{figure}[htbp]
	\centering
	\includegraphics[width=0.65\textwidth]{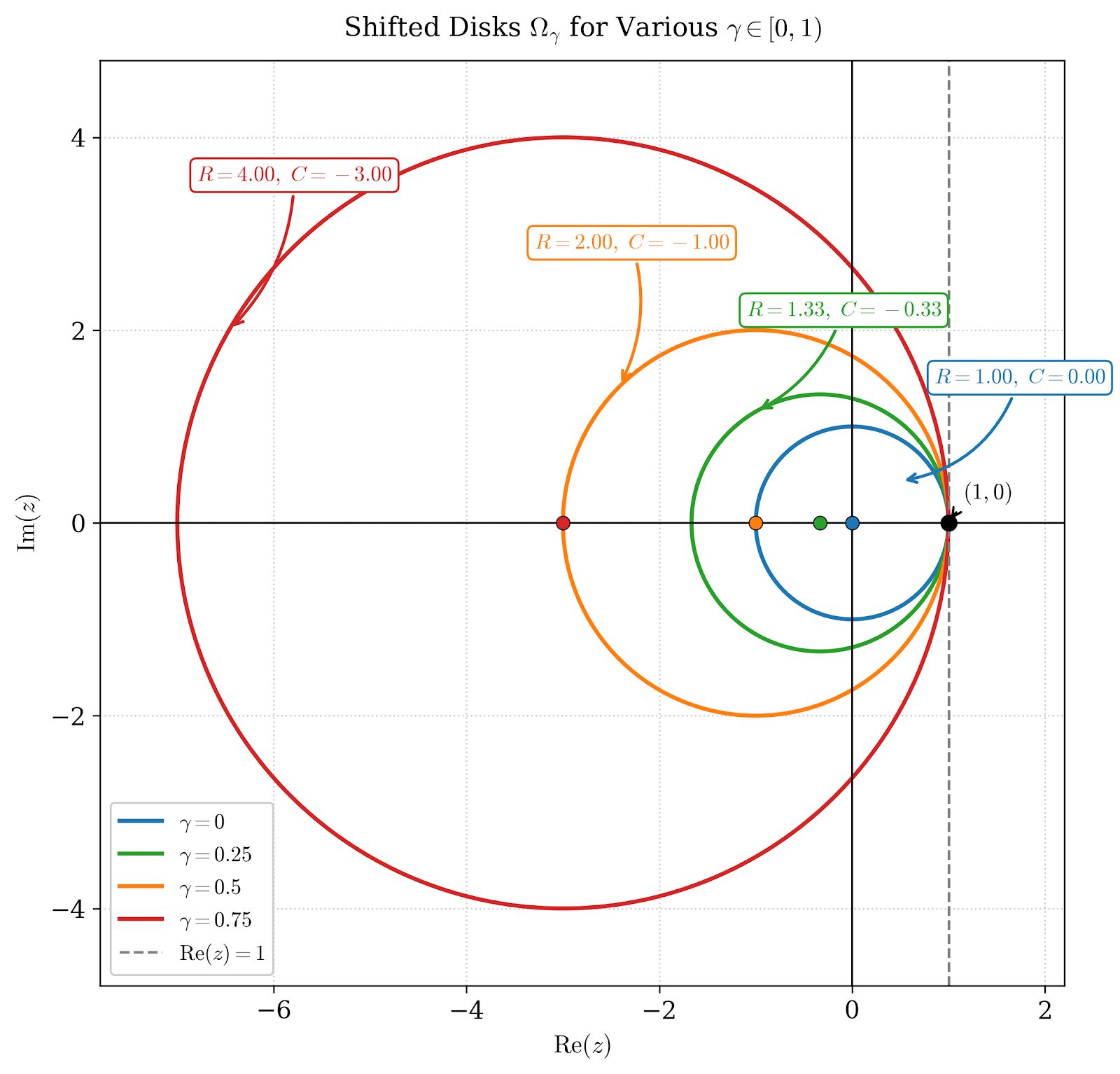}
	\caption{The shifted disks $\Omega_\gamma$ for various values of $\gamma$, sharing the common boundary point $(1,0)$. Each disk $\Omega_\gamma$ has a well-defined center $c(\gamma)$ and radius  $c(\gamma) = -\frac{\gamma}{1 - \gamma}$ and $R(\gamma) = \frac{1}{1 - \gamma}$, respectively. As $\gamma$ increases from $0$ towards $1$, the center $c(\gamma)$ moves leftward along the negative real axis towards $-\infty$, and the radius $R(\gamma)$ grows infinitely large.}
\end{figure}
The classical Bohr radius and its multi-dimensional analogues have been extensively investigated for the unit disk, the polydisk, and the unit ball of various Banach spaces (see, for instance, \cite{Boas-Khavinson-PAMS-1997} and references therein), the study of this geometric phenomenon over unbounded domains remains less developed.\vspace{1.2mm}
\begin{figure}[htbp]
	\centering
	\includegraphics[width=0.57\textwidth]{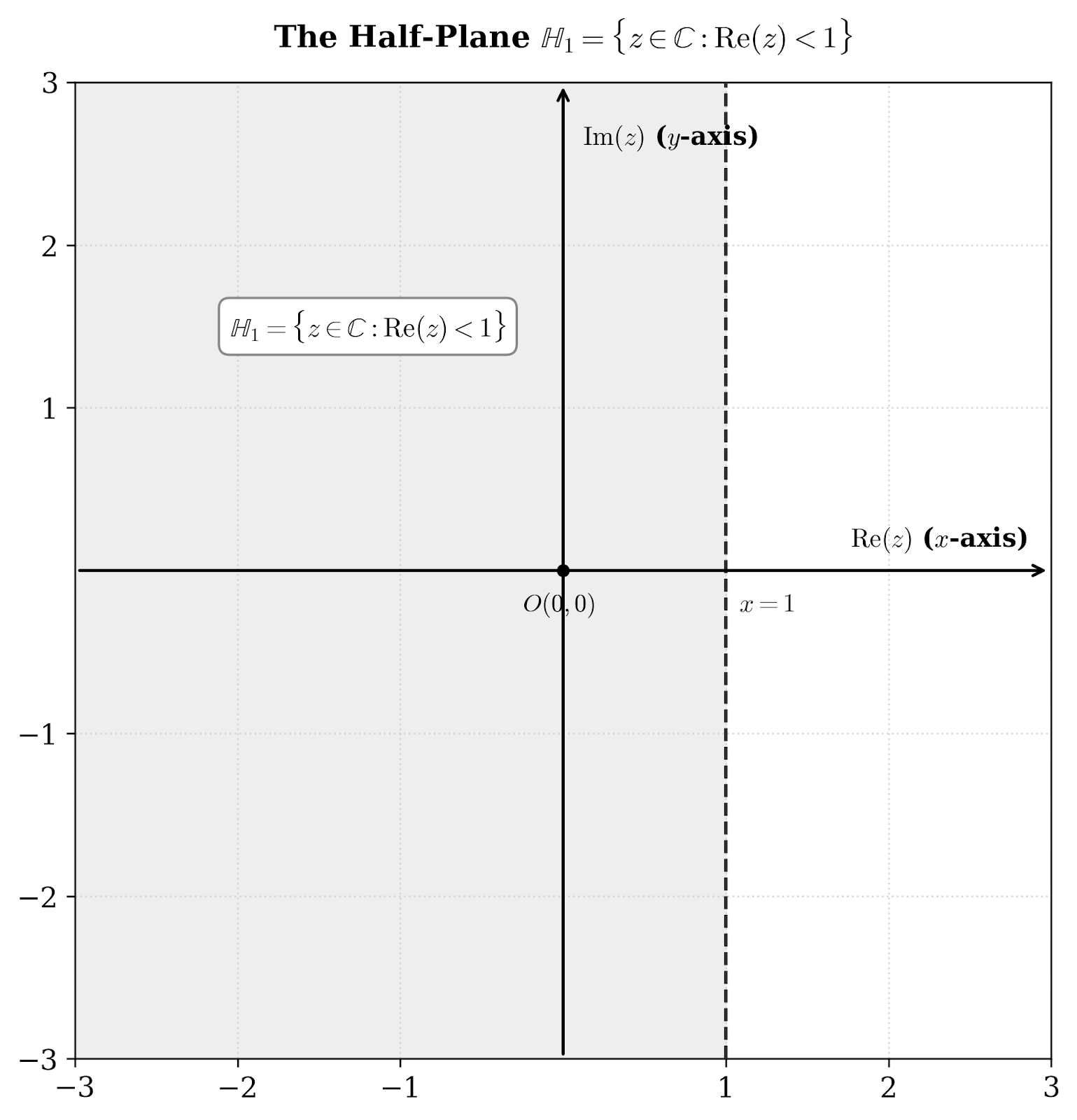}
	\caption{Representation of the open half-plane $\mathbb{H}_1 = \{ z \in \mathbb{C} : \text{Re}(z) < 1 \}$ in the complex plane. The shaded light-blue region represents all complex numbers with a real part strictly less than $1$. The dashed vertical line at $\text{Re}(z) = 1$ indicates that the boundary is excluded from the set, making $\mathbb{H}_1$ an open set.}
\end{figure}

In light of these foundational studies, it is natural to pose the following question regarding the asymptotic behavior of power series over asymmetric or unbounded regions.
	\begin{prob}What can be said about the Bohr radius for the class of analytic functions right-bounded half-plane $\mathbb{H}_1$?
\end{prob}
In the shifted disk $\Omega_\gamma$, {Evdoridis \emph{et al.} \cite{Evdoridis-Ponnusamy-Rasila-RM-2021}} studied improved versions of the classical Bohr inequality for analytic and harmonic mappings, with restrictions to the unit disk $\mathbb{D}$. Subsequently, Ahamed \emph{et al.} \cite{Ahamed-Allu-Halder-AFM-2022} established a number of sharp improved and refined versions of Bohr's inequality for analytic functions in $\Omega_\gamma$, also restricted to the unit disk $\mathbb{D}$.\vspace*{1.2mm}

In this paper, motivated by the work of {Fournier and Ruscheweyh \cite{Fournier-Ruscheweyh-CRMPLN-2010}}, we define the Bohr radius for the class $\mathcal{B}(\Omega_\gamma)$. It is defined as the number $\mathcal{B}_{\Omega_\gamma} \in (0, 1)$ such that
\[
\mathcal{B}_{\Omega_\gamma} = \sup \left\{ \rho \in (0, 1) : M_f(\rho) \le 1 \text{ for } f(z) = \sum_{n=0}^{\infty} \alpha_n \left( z + \frac{\gamma}{1 - \gamma} \right)^n \in \mathcal{B}(\Omega_\gamma), \, z \in \Omega_\gamma \right\},
\]
where 
\[
M_f(\rho) = \sum_{n=0}^{\infty} |\alpha_n| \left( \frac{\rho}{1 - \gamma} \right)^n\quad \text{with } |\gamma + (1 - \gamma)z| = \rho,
\]
is the majorant series associated with the analytic functions $f \in \mathcal{B}(\Omega_\gamma)$.\vspace*{2mm}

The primary objective of the present manuscript is to systematically generalize and establish refined variations of the classical Bohr inequality in the settings of the nested, internally tangent shifted disks $\Omega_\gamma$. By employing specialized conformal mapping methodologies and exploring the asymptotic behavior as the deformation parameter $\gamma \to 1^-$, we establish a rigorous framework to determine the Bohr radius for the right-bounded half-plane $\mathbb{H}_1$. Furthermore, this work focuses on establishing sharp Bohr-type majorant optimization bounds for unimodular bounded analytic functions that vanish at the shifted origin, thereby illuminating the structural invariance and delicate interplay between domain deformation and coefficient estimates in geometric function theory.\vspace{1.2mm}

To achieve these aims, the remainder of this paper is organized as follows. Section \ref{sec-2} provides essential auxiliary tools and preparatory lemmas, collecting foundational coefficient bounds and classical inequalities for analytic functions bounded in the unit disk. In Section \ref{sec-3}, we establish an improved majorant expansion technique for functions vanishing at the shifted origin, which serves as the base for our core results. Utilizing this, we systematically extend and improve the three classical coefficient inequalities originally obtained by Ponnusamy \emph{et al.} \cite{Ponnusamy-Vijayakumar-Wirths-RM-2020}. Specifically, we formulate and prove three general sharp Bohr-type inequalities (Theorems 3.1, 3.2, and 3.3) parameterized by the shifting index $\gamma$ and the zero multiplicity $k$, and discuss the corresponding sharp asymptotic radius via analytical and numerical evaluations of the underlying extremal polynomials.

\section{\bf Some preliminary results}\label{sec-2}
To establish our main results concerning the sharp Bohr-type radius bounds on the shifted disks $\Omega_\gamma$, we require several foundational lemmas and classical coefficient inequalities for the class $\mathcal{B}(\mathbb{D})$ of analytic self-maps of the unit disk {$\mathbb{D}$}. These auxiliary results provide the necessary estimates for the derivatives and Taylor coefficients of bounded analytic functions. We begin by recalling a classic distortion-type estimate for bounded analytic functions, which bounds the growth of high-order derivatives at any arbitrary point within the unit disk.
\begin{lemA}\label{Lem-A}\cite{Ruscheweyh-Serdica-1985}
	For $f \in \mathcal{B}(\mathbb{D})$, we have
	\[
	\frac{|f^{(n)}(\alpha)|}{n!} \le \frac{1 - |f(\alpha)|^2}{(1 - |\alpha|)^{n-1}(1 - |\alpha|^2)}
	\]
	for each $n \ge 1$ and $\alpha \in \mathbb{D}$.
\end{lemA}
Lemma B was introduced by Allu \emph{et al.} \cite{Allu-Biswas-Mandal-arXiv-2026} to establish results for functions in the shifted disk.
\begin{lemB}\label{Lem-B}\cite{Allu-Biswas-Mandal-arXiv-2026}
	Let $f$ be analytic in $\Omega_\gamma$, bounded by $1$ with the series expansion 
	$$f(z) = \sum_{n=0}^{\infty} a_n \left( z+\frac{\gamma}{1 - \gamma} \right)^n$$ 
	in $\Omega_\gamma$. Then, $|a_n| \le (1 - \gamma)^n (1 - |a_0|^2)$ for $n \ge 1$.
\end{lemB}
Our principal aim is to compare the majorant series $\sum_{n=0}^{\infty} |a_{n}| r^{n}$ with another classical structural functional frequently considered in geometric function theory, namely the square of the $H^2$-norm defined by 
$$ \|f\|_{r}^{2} = \sum_{n=0}^{\infty} |a_{n}|^{2} r^{2n}. $$
In \cite{Ponnusamy-Vijayakumar-Wirths-RM-2020}, Ponnusamy \emph{et al.} provided a sharp refinement of the classical inequality $\|f\|_{r}^{2} \le 1$ for $f \in \mathcal{B}(\mathbb{D})$ by explicitly embedding this $H^2$-norm into the upper bound of the corresponding majorant series. This foundational formulation is stated below.
\begin{thmB}\cite[Theorem 1]{Ponnusamy-Vijayakumar-Wirths-RM-2020}
	Suppose that $f \in \mathcal{B}$ and $f(z) = \sum_{n=0}^{\infty} a_n z^n$. Then the inequality
	$$\sum_{n=0}^{\infty} |a_n| r^n \le \frac{1}{1 - r} \left( 1 - r \|f\|^2_{r} \right)$$
	is valid for $r \in [0, 1)$. Equality is attained for $f(z) = 1$, $z \in \mathbb{D}$.
\end{thmB}
\begin{thmC}\cite[Theorem 2]{Ponnusamy-Vijayakumar-Wirths-RM-2020}
	Suppose that $f \in \mathcal{B}$, $f(z) = \sum_{n=0}^{\infty} a_n z^n$, and $f_0(z) = f(z) - a_0$. Then
	\begin{equation}\label{Eq-22.1}
		\sum_{n=0}^{\infty} |a_n| r^n + \left( \frac{1}{1 + |a_0|} + \frac{r}{1 - r} \right) \|f_0\|^2_{r} \le 1 \quad \text{for } r \le \frac{1}{2 + |a_0|}
	\end{equation}
	and the numbers $\frac{1}{2+|a_0|}$ and $\frac{1}{1+|a_0|}$ cannot be improved. Moreover,
	\begin{equation}\label{Eq-22.2}
		|a_0|^2 + \sum_{n=1}^{\infty} |a_n| r^n + \left( \frac{1}{1 + |a_0|} + \frac{r}{1 - r} \right) \|f_0\|^2_{r} \le 1 \quad \text{for } r \le \frac{1}{2}
	\end{equation}
	and the numbers $\frac{1}{2}$ and $\frac{1}{1+|a_0|}$ cannot be improved.
\end{thmC}
The authors have remarked that $\frac{1}{3} \le \frac{1}{2+|a_0|} \le \frac{1}{2}$. In particular, in the case of $a_0 = 0$ the conclusion of Theorem C, namely, the inequality \eqref{Eq-22.1} gives that
\begin{align*}
	\sum_{n=1}^{\infty} |a_n| r^n + \frac{1}{1 - r} \|f\|^2_{r} \le 1 \quad \text{for }\text{for } r \le \frac{1}{2}.
\end{align*}
\begin{thmD}\cite[Theorem 3]{Ponnusamy-Vijayakumar-Wirths-RM-2020}
	Suppose that $f \in \mathcal{B}$ and $f(z) = \sum_{n=1}^{\infty} a_n z^n$. Then we have the following:\vspace*{1.2mm}
	
	\item [(a)] $\displaystyle\sum_{n=1}^{\infty} |a_n| r^n + \left( \frac{1}{1 + |a_1|} + \frac{r}{1 - r} \right) \sum_{n=2}^{\infty} |a_n|^2 r^{2n-1} \le 1 \quad \text{for } r \le \frac{3}{5}.$\vspace{1.2mm}
	
	\noindent The number $3/5$ is sharp. \vspace*{1.2mm}
	
	\item [(b)]  $\displaystyle\sum_{n=1}^{\infty} |a_n| r^n + \left( \frac{r^{-1}}{1 + |a_1|} + \frac{1}{1 - r} \right) \|f\|^2_{r} \le 1 \quad \text{for } r \le \frac{5 - \sqrt{17}}{2}.$\vspace{1.2mm}
	
	\noindent The number $\frac{5-\sqrt{17}}{2}$ is sharp.\vspace*{1.2mm}
	
	\item [(c)]  $\displaystyle\sum_{n=1}^{\infty} |a_n| r^n + \left( \frac{r^{-1}}{1 + |a_1|} + \frac{1}{1 - r} \right) \|f\|^2_{r} \le 1 \quad \text{for } r \le r(a),$	where 
	$$r(a) = \frac{2(1 + a)}{1 + 2a + 2a^2 + \sqrt{4a^4 + 8a + 5}}.$$
	The radius $r(a)$ is sharp for any $a \in [0, 1)$.
\end{thmD}
\begin{thmE}\cite[Theorem 1]{Ponnusamy-Vijayakumar-Wirths-RM-2020}
	Suppose that $f \in \mathcal{B}$ and $f(z) = \sum_{n=0}^{\infty} a_n z^n$. Then the inequality
	\[
	\sum_{n=0}^{\infty} |a_n| r^n \le \frac{1}{1-r} \left( 1 - r \|f\|_{r}^2 \right)
	\]
	is valid for $r \in [0, 1)$. Equality is attained for $f(z) = 1\;{\mbox{for}}\; z \in \mathbb{D}$.
\end{thmE}
\begin{rem}
	It is remarked in \cite{Ponnusamy-Vijayakumar-Wirths-RM-2020} that
	$$\frac{3}{5} > \frac{1}{2} = r\left(\frac{1}{\sqrt{2}}\right) > \frac{5 - \sqrt{17}}{2}.$$
\end{rem}
{In summary, the principal objective of this paper is to provide a sharp geometric description of the shifted disk family $\Omega_\gamma$ together with its limiting half-plane $\mathbb{H}_1 = \{ z \in \mathbb{C} : \text{Re}(z) < 1 \}$. We establish that as $\gamma \to 1^-$, the nested sequence $\Omega_\gamma$ exhausts $\mathbb{H}_1$ monotonically, converging locally in the Hausdorff metric. The formal proofs and their analytic ramifications form the subject of the remaining sections.}
\section{\bf Improved Bohr inequality for functions in shifted disk $\Omega_\gamma$}\label{sec-3}
In this section, we establish a generalized lemma providing sharp coefficient estimates for analytic functions of the form 
\begin{align}\label{Eq-33.1}
	f(z) = \sum_{n=k}^{\infty} a_n \left(z+\frac{\gamma}{1-\gamma}\right)^n
\end{align} bounded by $1$ on the shifted disk $\Omega_{\gamma}$. For the particular case $k=0$, our result reduces to \cite[Lemma 4]{Allu-Biswas-Mandal-arXiv-2026}. However, while the authors in \cite{Allu-Biswas-Mandal-arXiv-2026} omitted the sharp characterization of their coefficient bounds, we explicitly demonstrate the sharpness of the inequality in a more general setting for arbitrary indices $k \in \mathbb{N}_0$. Our result thus provides a significant refinement and a complete extension of the existing literature.
\begin{lem}\label{Lem-3.1}
	Let $k \in \mathbb{N}_0$ and let $f$ be an analytic function on the shifted disk $\Omega_\gamma$ bounded by $1$ in modulus, with the power series expansion \eqref{Eq-33.1}. Then, 
	\begin{align}\label{Eq-33.2}|a_n|\leq(1-\gamma)^{n}(1-|a_k|^2) \;\text{for}\; n\geq k+1.
	\end{align}Moreover, the inequality \eqref{Eq-33.2} is sharp for each $n \geq k+1$.
	\end{lem}
	\begin{proof}[\bf Proof of Lemma \ref{Lem-3.1}]
		{Let} $\Phi:\mathbb{D}\rightarrow\Omega_\gamma$ {be the conformal map} defined by
		\begin{align*}\Phi(\xi)=\frac{\xi-\gamma}{1-\gamma}, \quad \xi \in \mathbb{D}.
		\end{align*}It is clear that $\Phi$ maps {the unit disk} $\mathbb{D}$ univalently onto $\Omega_\gamma$. Consequently, the composition $F = f\circ \Phi$ is an analytic map from $\mathbb{D}$ into $\mathbb{D}$ satisfying $|F(\xi)| \le 1$ for all $\xi \in \mathbb{D}$. For any $z \in \Omega_\gamma$, the variable $\xi = \gamma + (1-\gamma)z$ belongs to $\mathbb{D}$. Substituting $\Phi(\xi) = z$, we obtain the {following} Taylor expansion of $F$ about the origin
		\begin{align*}
			F(\xi) = f\left(\frac{\xi-\gamma}{1-\gamma}\right) = \sum_{n=k}^{\infty}\frac{a_n}{(1-\gamma)^n}\xi^n = \xi^k \sum_{m=0}^{\infty}\frac{a_{m+k}}{(1-\gamma)^{m+k}}\xi^{m}.\end{align*}We define the auxiliary function $g: \mathbb{D} \to \mathbb{C}$ by $g(\xi) = \sum_{m=0}^{\infty} b_m \xi^m$, where
			\begin{align*}b_m = \frac{a_{m+k}}{(1-\gamma)^{m+k}},\; m \geq 0.
			\end{align*}Since $F(\xi) = \xi^k g(\xi)$ and $|F(\xi)| \le 1$, {in view of} Schwarz's lemma, it is easy to see that $g$ is analytic in $\mathbb{D}$ and satisfies $|g(\xi)| \le 1$ for all $\xi \in \mathbb{D}$. Applying the classical coefficient inequality for bounded analytic functions (Wiener's estimate or the {consequence of }Schwarz-Pick lemma), the coefficients of $g$ satisfy\begin{align*}|b_m| \leq 1 - |b_0|^2 \quad \text{for } m \geq 1.\end{align*}
		Substituting $b_0 = {a_k}/{(1-\gamma)^k}$ and $b_m = {a_{m+k}}/{(1-\gamma)^{m+k}}$, we immediately deduce that\begin{align*}\frac{|a_{m+k}|}{(1-\gamma)^{m+k}} \leq 1 - \frac{|a_k|^2}{(1-\gamma)^{2k}} \le 1 - |a_k|^2,
	\end{align*}where we note that $(1-\gamma)^{2k} \le 1$ since $\gamma \in [0,1)$. Setting $n = m+k \ge k+1$, a direct simplification {gives}
	\begin{align*}
		|a_{n}|\leq(1-\gamma)^n(1-|a_k|^2),
	\end{align*}
	which completes the validation of {the} inequality \eqref{Eq-33.2}.\vspace{1.2mm}
	
	To establish the sharpness of the inequality \eqref{Eq-33.2}, fix an index $n \ge k+1$ and consider the function $f_0(z) = (G \circ \Phi^{-1})(z)$, where 
	\begin{align*}
		G(\xi) = \xi^k \left( \frac{a - \xi^{n-k}}{1 - a\xi^{n-k}} \right)\; \mbox{for}\;a \in [0,1).
	\end{align*} Since $\Phi^{-1}(z) = \gamma + (1-\gamma)z$, we have
	\begin{align*}f_0(z) &= {\left(\gamma + (1-\gamma)z\right)^k} \left( a - (1-a^2)\sum_{j=1}^{\infty} a^{j-1} {\left(\gamma + (1-\gamma)z\right)}^{j(n-k)} \right) \\&= a(1-\gamma)^k \left(z+\frac{\gamma}{\text{1}-\gamma}\right)^k - (1-a^2)(1-\gamma)^n \left(z+\frac{\gamma}{1-\gamma}\right)^n + \cdots.
	\end{align*}
	Thus, it follows that $a_k = a(1-\gamma)^k$ and $a_n =-(1-a^2)(1-\gamma)^n$. We {obtain} that $|a_n| = (1-|a_k|^2)(1-\gamma)^n$ for $n\geq k+1$, proving that the inequality is sharp.
	\end{proof}
	For a further study of Theorem D, it is natural to raise the following.
	\begin{prob}\label{P-2.1}
		Can the Bohr-type inequalities in Theorem D be generalized to the shifted disk $\Omega_\gamma$ for functions of the form $f(z) = \sum_{n=k}^{\infty} a_n \left(z + \frac{\gamma}{1-\gamma}\right)^n$?
	\end{prob}
	In this paper, one of our {aims} is to give {affirmative} answer to the Problem \ref{P-2.1}. By {the} virtue of Lemma \ref{Lem-3.1}, we establish a sharp refinement and generalization of item (a) of Theorem D by extending the underlying configuration to the general setting of the shifted disk $\Omega_{\gamma}$.
\begin{thm}\label{thm-3.1}
	Suppose that $f \in \mathcal{B}(\Omega_\gamma)$ and $f$ is given by \eqref{Eq-33.1}. Then 
	\begin{align*}
		\sum_{n=k}^{\infty}\frac{|a_n|}{(1-\gamma)^{n}}\rho^n+\left(\frac{(1-\gamma)^k}{(1-\gamma)^k+|a_k|}+\frac{\rho}{1-\rho}\right)\sum_{n=k+1}^{\infty}\frac{|a_n|^2}{(1-\gamma)^{2n}}\rho^{2n-k}\leq1,
	\end{align*}
	for $\rho\leq\rho_k$, where $\rho_k$ is the root in $(0,1)$ of the equation
	\begin{align}\label{Eq-33.3}
		5\rho^{k+1}-2\rho^k+\rho^{k-1}+4\rho-4=0.
	\end{align}
	The radius $ \rho_k $ is best possible.
\end{thm}
\begin{rem}
	In particular, by setting $k=1$ and $\gamma=0$, we obtain the result \cite[Theorem 3]{Ponnusamy-Vijayakumar-Wirths-RM-2020}. 
\end{rem}
\begin{cor}
	Suppose that $f \in \mathcal{B}$ and $f(z) = \sum_{n=1}^{\infty} a_n z^n$. Then 
	\begin{align*}
		\sum_{n=1}^{\infty}|a_n|\rho^n + \left(\frac{1}{1+|a_1|}+\frac{\rho}{1-\rho}\right)\sum_{n=2}^{\infty}|a_n|^2\rho^{2n-1} \leq 1,
	\end{align*}
	for $\rho \leq \rho_1$, where $\rho_1 =3/5$ is the unique positive root in $(0,1)$ of the quadratic equation $5\rho^2 + 2\rho - 3 = 0.$
\end{cor}
\begin{rem}
	In particular, by setting $k=0$, our result structurally reduces to the coefficient inequality established by Allu \emph{et al.} \cite[Theorem 1]{Allu-Biswas-Mandal-arXiv-2026}.
\end{rem}
\begin{cor}
	Suppose that $f \in \mathcal{B}(\Omega_\gamma)$ and $f$ is given by \eqref{Eq-33.1}. Then 
	\begin{align*}
		\sum_{n=0}^{\infty}\frac{|a_n|}{(1-\gamma)^{n}}\rho^n+\left(\frac{1}{1+|a_0|}+\frac{\rho}{1-\rho}\right)\sum_{n=1}^{\infty}\frac{|a_n|^2}{(1-\gamma)^{2n}}\rho^{2n}\leq1,
	\end{align*}
	for $\rho\leq \rho_0$, where $\rho_0 =1/3$ is the unique root in $(0,1)$ of the equation $9\rho^2 - 6\rho + 1 = 0.$
\end{cor}

\begin{proof}[\bf Proof of Theorem \ref{thm-3.1}]
    Let
    \[
    f(z)=\sum_{n=k}^{\infty}a_n\left(z+\frac{\gamma}{1+\gamma}\right)^n,
    \]
    where $f$ is analytic in $\Omega_\gamma$ and satisfies $|f(z)|\le1$ for all $z\in\Omega_\gamma$. Let $\Phi:\mathbb{D}\rightarrow\Omega_\gamma$ defined by 
    \begin{align*}
    	\Phi(z)=\frac{z-\gamma}{1-\gamma}\quad\mbox{for}\;z\in\mathbb{D}.
    \end{align*}
    Since $\Phi:\mathbb{D}\to\Omega_\gamma$ is conformal {map}, the composition
    $f\circ\Phi:\mathbb{D}\to\mathbb{D}$ is analytic {function}. Moreover, for every
    $z\in\Omega_\gamma$, the point $\xi=\gamma+(1-\gamma)z$ belongs to $\mathbb{D}$. Consequently,
    \begin{align*}
    	(f\circ\Phi)(\xi)=f\left(\frac{\xi-\gamma}{1-\gamma}\right)=\sum_{n=k}^{\infty}\frac{a_n}{(1-\gamma)^n}\xi^n=\xi^k\sum_{n=0}^{\infty}\frac{a_{n+k}}{(1-\gamma)^{n+k}}\xi^{n}=\xi^kg(\xi),
    \end{align*}
    where
    \begin{align*}
    	g(\xi)=\sum_{n=0}^{\infty}b_n\xi^n\;\;{ \mbox{where}}\;b_n=\frac{a_{n+k}}{(1-\gamma)^{n+k}}.
    \end{align*}
    Since $f\circ\Phi$ is analytic, it follows that $g$ is analytic in
    $\mathbb{D}$ and satisfies $|g(z)|\le1$ for all $z\in\mathbb{D}$.
    Applying inequality~(4) from the proof of Theorem E, we obtain
    \begin{align}\label{eq-3.1}
    	\sum_{n=0}^{\infty}|b_n|\rho^{n}\leq |b_0|+\frac{\rho}{1-\rho}(1-|b_0|^2)-\left(\frac{1}{1+|b_0|}+\frac{\rho}{1-\rho}\right)||g_0||^2_{\rho}.
    \end{align}
    Using the inequality \eqref{eq-3.1}, we obtain
    \begin{align*}
    	\sum_{n=k}^{\infty}\frac{|a_n|}{(1-\gamma)^n}\rho^n&=\sum_{n=0}^{\infty}|b_n|\rho^{n+k}\\&\leq \rho^k\left(|b_0|+\frac{\rho}{1-\rho}(1-|b_0|^2)-\left(\frac{1}{1+|b_0|}+\frac{\rho}{1-\rho}\right)||g_0||^2_r\right).
    \end{align*}
    This implies that
    \begin{align*}
    	\sum_{n=k}^{\infty}\frac{|a_n|}{(1-\gamma)^n}\rho^n&+\left(\frac{(1-\gamma)^k}{(1-\gamma)^k+|a_k|}+\frac{\rho}{1-\rho}\right)\sum_{n=k+1}^{\infty}\frac{|a_n|^2}{(1-\gamma)^{2n}}\rho^{2n-k}\\&\leq \rho^k\left(|b_0|+\frac{\rho}{1-\rho}(1-|b_0|^2)\right)=:\psi_{\rho}(|b_0|),
    \end{align*}
     where
    \begin{align*}
    	\psi_{\rho}(x)=\rho^k\left(x+\frac{\rho}{1-\rho}(1-x^2)\right)\;\mbox{for}\;|b_0|=x\in[0,1],\;\mbox{and}\;\rho\in[0,1).
    \end{align*}
   
    We just need to maximize $\psi(x,\rho)$ (with respect to $x$) over the interval $[0, 1]$. We see that $\psi$ has a critical point at $x_0 = (1-\rho)/(2\rho)$ and obtain that the maximum occurs at this point, so that
    \begin{align*}
    	\psi_{\rho}(x)\le \psi_{\rho}(x_0) &= \psi\left(\frac{1-\rho}{2\rho},\rho\right)\\&=\rho^k\left(\frac{1-\rho}{2\rho}+\frac{\rho}{1-\rho}\left(1-\frac{(1-\rho)^2}{4\rho^2}\right)\right)\\&=\rho^{k-1}\left(\frac{1-\rho}{2}+\frac{(3\rho-1)(1+\rho)}{4(1-\rho)}\right)\\&:=1+P_k(\rho),
    \end{align*}
    where 
    \begin{align*}
    	P_k(\rho)=\rho^{k-1}\left(\frac{1-\rho}{2}+\frac{(3\rho-1)(1+\rho)}{4(1-\rho)}\right)-1.
    \end{align*}
    To complete the proof, it suffices to establish that $P_k(\rho)\le0$,
    since this immediately yields $\psi_\rho(x)\le1$. Observe that
    $P_k$ is increasing on $(0,1)$, satisfies $P_k(0)=-1$, and
    \[
    \lim_{\rho\to1^-}P_k(\rho)=+\infty.
    \]
    Hence, $P_k(\rho)=0$ possesses a unique solution
    $\rho_k$ {in }$(0,1)$, and therefore $P_k(\rho)\le0,\;\text{whenever } \rho\le\rho_k,$ where $\rho_k$ is the unique positive root of
    \begin{align*}
    	5\rho^{k+1}-2\rho^k+\rho^{k-1}+4\rho-4=0.
    \end{align*}
    To establish the sharpness, let $f=f_a\circ\phi,$ where $\phi:\Omega_\gamma\to\mathbb{D}$ is given by
    $\phi(z)=\gamma+(1-\gamma)z$, and
    \[
    f_a(z)=z^k\left(\frac{a-z}{1-az}\right),\; a\in(0,1).
    \]
    Since $\phi$ is a conformal mapping of $\Omega_\gamma$ onto $\mathbb{D}$,
    it follows that $f$ maps $\Omega_\gamma$ univalently onto $\mathbb{D}$.
    Hence,
    \begin{align*}
    	f(z)=(1-\gamma)^k\left(z+\frac{\gamma}{1-\gamma}\right)^k\left(\frac{a-(1-\gamma)\left(z+\frac{\gamma}{1-\gamma}\right)}{1-a(1-\gamma)\left(z+\frac{\gamma}{1-\gamma}\right)}\right)=\sum_{n=k}^{\infty}C_n\left(z+\frac{\gamma}{1-\gamma}\right)^n,
    \end{align*}
    where 
    \begin{align*}
    	C_k=a(1-\gamma)^k\;\mbox{and}\; C_n=-(1-\gamma)^na^{n-k-1}(1-a^2)
    \end{align*}
     for $n\geq k+1$. For this function, we obtain
    \begin{align*}
    	&\sum_{n=k}^{\infty}\frac{|C_n|}{(1-\gamma)^{n}}\rho^n+\left(\frac{(1-\gamma)^k}{(1-\gamma)^k+|C_k|}+\frac{\rho}{1-\rho}\right)\sum_{n=k+1}^{\infty}\frac{|C_n|^2}{(1-\gamma)^{2n}}\rho^{2n-k}\\&=a\rho^k+(1-a^2)\sum_{n=k+1}^{\infty}a^{n-k-1}\rho^n +\left(\frac{1}{1+a}+\frac{\rho}{1-\rho}\right)(1-a^2)^2\sum_{n=k+1}^{\infty}a^{2(n-k-1)}\rho^{2n-k}\\&=\rho^k\left(a+\frac{(1-a^2)\rho}{1-a\rho}+\frac{(1-a^2)(1-a)\rho^2}{(1-\rho)(1-a\rho)}\right)\\&=\rho^k\left(a+\frac{(1-a^2)\rho}{1-\rho}\right)\\&:=F_1(a,\rho),
    \end{align*}
    where 
    \begin{align*}
    	F_1(a,\rho)=\rho^k\left(a+\frac{(1-a^2)\rho}{1-\rho}\right).
    \end{align*}
    It is readily verified that $F_1(a,\rho)>1$ for every $\rho>\rho_k$ as
    \[
    a\to\frac{1-\rho}{2\rho},
    \]
    where $\rho\in(1/3,1)$. This establishes the sharpness of the result and completes the proof.
    \end{proof}  
    
    \begin{figure}[htbp]
    	\centering
    	\includegraphics[width=0.95\textwidth]{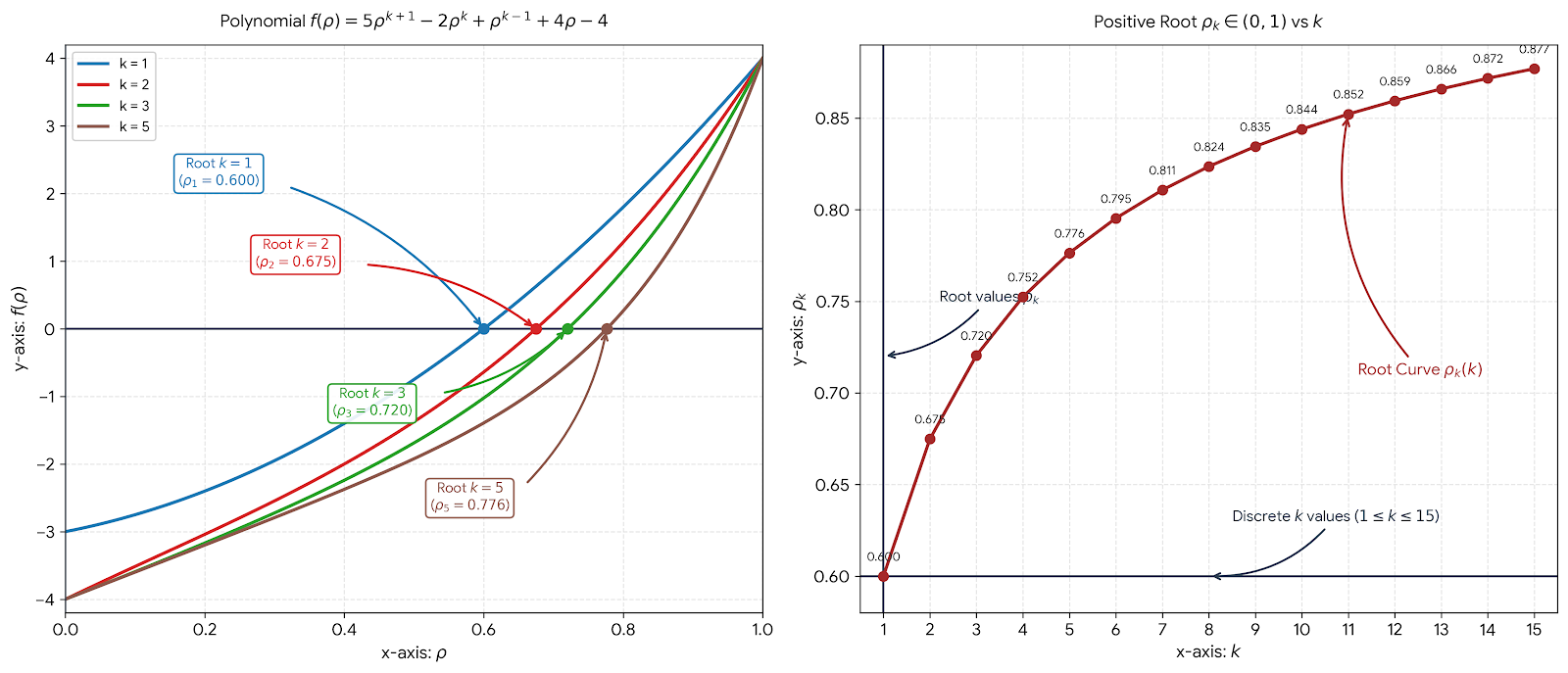}
    	\caption{Analysis of the polynomial $f(\rho)$ (left) and the trajectory of its unique positive root $\rho_k$ across values of $k$ (right).}
    	\label{fig:theorem_visuals}
    \end{figure}
    We obtain the following sharp Bohr-type inequality for functions in half-plan $\mathbb{H}_1 = \{ w \in \mathbb{C} : \operatorname{Re}(w) < 1 \}$.
    \begin{thm}\label{thm-3.1-halfplane}
    	Let $\mathbb{H}_1 = \{ w \in \mathbb{C} : \operatorname{Re}(w) < 1 \}$ and suppose that $f \in \mathcal{B}(\mathbb{H}_1)$ possesses the Taylor expansion 
    	\begin{align}\label{Eq-halfplane-exp}
    		f(w) = \sum_{n=k}^{\infty} A_n \left( w + 1 \right)^n, \quad w \in \mathbb{H}_1.
    	\end{align}\label{Eq-33.4}
    	Then, the sharp inequality
    	\begin{align*}
    		\sum_{n=k}^{\infty} |A_n| \rho^n + \left(\frac{1}{1+|A_k|} + \frac{\rho}{1-\rho}\right) \sum_{n=k+1}^{\infty} |A_n|^2 \rho^{2n-k} \le 1,
    	\end{align*}
    	holds for $\rho \le \rho_k$, where $\rho_k$ is the unique positive root in $(0,1)$ of the algebraic equation \eqref{Eq-33.3}. The radius $ \rho_k $ is {the }best possible.
    \end{thm}
    \begin{proof}[\bf Proof of Theorem \ref{thm-3.1-halfplane}]
    	The right-bounded half-plane $\mathbb{H}_1 = \{ w \in \mathbb{C} : \operatorname{Re}(w) < 1 \}$ is established as the limit of a family of nested, internally tangent shifted disks $\Omega_{\gamma}$ with $\gamma \in [0,1)$.\vspace{1.2mm}
    	
    	As $\gamma \to 1^-$, the family of domains asymptotically deforms and expands to fill the entire half-plane $\mathbb{H}_1$. Consider an analytic function $f \in \mathcal{B}(\mathbb{H}_1)$ possessing the Taylor expansion about the shifted boundary anchor point:  
    	\[
    	f(w) = \sum_{n=k}^{\infty} A_n (w + 1)^n.
    	\]
    	Because $\Omega_{\gamma} \subsetneq \mathbb{H}_1$ for all $\gamma \in [0,1)$, the function $f$ restricted to any shifted disk $\Omega_{\gamma}$ belongs to $\mathcal{B}(\Omega_{\gamma})$. Thus, $f(w)$ matches the formulation of Theorem \ref{thm-3.1} given {by}  
    	\[
    	f(z) = \sum_{n=k}^{\infty} a_n \left( z + \frac{\gamma}{1-\gamma} \right)^n.
    	\]
    	Comparing the coefficients yields the identity relation  
    	\[
    	a_n = A_n (1-\gamma)^n \;\mbox{if, and only if, } A_n = \frac{a_n}{(1-\gamma)^n}.
    	\]
    	Applying Theorem \ref{thm-3.1} to $f \in \mathcal{B}(\Omega_{\gamma})$ yields 
    	\begin{align*}
    		\sum_{n=k}^{\infty} \frac{|a_{n}|}{(1-\gamma)^{n}} \rho^{n} + \left( \frac{(1-\gamma)^{k}}{(1-\gamma)^{k}+|a_{k}|} + \frac{\rho}{1-\rho} \right) \sum_{n=k+1}^{\infty} \frac{|a_{n}|^{2}}{(1-\gamma)^{2n}} \rho^{2n-k} \le 1.
    	\end{align*}
    	Substituting the coefficient conversion identity $A_n = {a_n}/{(1-\gamma)^n}$ into this inequality directly { produces}
    	\begin{align*}
    		\sum_{n=k}^{\infty} |A_n| \rho^n + \left( \frac{1}{1+|A_k|} + \frac{\rho}{1-\rho} \right) \sum_{n=k+1}^{\infty} |A_n|^2 \rho^{2n-k} \le 1.
    	\end{align*}
    	This inequality holds for all $\rho \le \rho_k$, where $\rho_k \in (0,1)$ is the unique positive root of the algebraic constraint equation
    	\begin{align*}
    		5\rho^{k+1}-2\rho^{k}+\rho^{k-1}+4\rho-4=0.
    	\end{align*}
    	The sharpness of the radius $\rho_k$ translates directly from the optimality verified in Theorem \ref{thm-3.1} via the asymptotic behavior of the testing function  
    	\begin{align*}
    		f(z) &= {\left(\gamma + (1-\gamma)z\right)}^k \left( \frac{a - {\left(\gamma + (1-\gamma)z\right)}}{1 - a{\left(\gamma + (1-\gamma)z\right)}} \right)\\&= a \left(\gamma + (1-\gamma)z\right)^k + (a^2 - 1) \sum_{m=1}^\infty a^{m-1} \left(\gamma + (1-\gamma)z\right)^{m+k}.
    	\end{align*}
    	Extracting the coefficient of $z^n$ from the expansion of $f$ gives that 
    	\begin{align*}
    		c_n = (1-\gamma)^n \left( a \binom{k}{n} \gamma^{k-n} + (a^2 - 1) \sum_{m=\max(1, n-k)}^\infty a^{m-1} \binom{m+k}{n} \gamma^{m+k-n} \right).
    	\end{align*}
    	As $\gamma \to 1^-$, the geometric bound is preserved, confirming that $\rho_k$ is the best possible radius.
    \end{proof}
By applying Lemma \ref{Lem-3.1}, we provide a significant refinement and extension of item (b) of Theorem D by translating the problem into the generalized framework of the shifted disk $\Omega_{\gamma}$.
\begin{thm}\label{thm-3.2}
	Suppose that $f \in \mathcal{B}(\Omega_\gamma)$ and $f$ is given by \eqref{Eq-33.1}. Then 
	\begin{align*}
		\sum_{n=k}^{\infty}\frac{|a_n|}{(1-\gamma)^{n}}\rho^n+\left(\frac{r^{-k}}{1+a}+\frac{\rho^{1-k}}{1-\rho}\right)||f||_\rho^2\leq1,
	\end{align*}
	for $\rho\leq\rho^*_k$, where $\rho^*_k$ is the positive root in $(0,1)$ of the equation
	\begin{align}\label{Eq-33.6}
	\rho^{k+1}-3\rho^k-2\rho+2=0.
	\end{align}
	The radius $\rho^*_k$ is best possible.
\end{thm}
\begin{rem}
	In particular, by setting $k=1$ and $\gamma=0$, our result reduces to the {result} established by Ponnusamy \emph{et al.} \cite[Theorem 3]{Ponnusamy-Vijayakumar-Wirths-RM-2020} ({\textit{i.e.,}} Theorem D).
\end{rem}
\begin{cor}
	Suppose that $f \in \mathcal{B}$ and $f(z) = \sum_{n=1}^{\infty} a_n z^n$. Then 
	\begin{align*}
		\sum_{n=1}^{\infty}|a_n|\rho^n + \left(\frac{\rho^{-1}}{1+|a_1|}+\frac{1}{1-\rho}\right) \|f\|^2_{\rho} \leq 1,
	\end{align*}
	for $\rho \leq \rho_0:={(5-\sqrt{17})}/{2}$, where $\rho_0 $ is the unique positive root in $(0,1)$ of the quadratic equation $\rho^2 - 5\rho + 2 = 0.$
\end{cor}
 {As an immediate geometric consequence, we obtain the following sharp refined Bohr-type inequality embedded with the $H^2$-norm.}
\begin{thm}\label{thm-3.2-halfplane}
	Let $\mathbb{H}_1 = \{ w \in \mathbb{C} : \operatorname{Re}(w) < 1 \}$ and assume that $f \in \mathcal{B}(\mathbb{H}_1)$ has the series representation \eqref{Eq-33.4}.	Then, 
	\begin{align*}
		\sum_{n=k}^{\infty} |A_n| \rho^n + \left( \frac{\rho^{-k}}{1 + |A_k|} + \frac{\rho^{1-k}}{1 - \rho} \right) \|f\|_\rho^2 \le 1,
	\end{align*}
	for $\rho \le \rho^*_k$, where $\|f\|_\rho^2 = \sum_{n=k}^{\infty} |A_n|^2 \rho^{2n}$, and $\rho^*_k$ is the unique positive root in $(0,1)$ of the equation \eqref{Eq-33.6}. The radius $\rho^*_k$ is best possible.
\end{thm}
\begin{proof}[\bf Proof of Theorem \ref{thm-3.2-halfplane}]
	{The half-plane $\mathbb{H}_1 = \{ w \in \mathbb{C} : \operatorname{Re}(w) < 1 \}$ is the limit of a family of nested, shifted disks $\Omega_{\gamma}$. In deed, $\gamma \to 1^-$, the family of domains asymptotically deforms and expands to fill the entire half-plane $\mathbb{H}_1$.}  \vspace{1.2mm}
	
	Consider an analytic function $f \in \mathcal{B}(\mathbb{H}_1)$ possessing the Taylor expansion about the shifted boundary 
	\[
	f(w) = \sum_{n=k}^{\infty} A_n (w + 1)^n.
	\]
	Because $\Omega_{\gamma} \subsetneq \mathbb{H}_1$ for all $\gamma \in [0,1)$, the function $f$ restricted to any shifted disk $\Omega_{\gamma}$ belongs to $\mathcal{B}(\Omega_{\gamma})$. Thus, $f(w)$ matches the formulation of Theorem 3.3 given by 
	\[
	f(z) = \sum_{n=k}^{\infty} a_n \left( z + \frac{\gamma}{1-\gamma} \right)^n.
	\]
	Comparing the coefficients yields the identity relation  
\begin{align*}
	a_n = A_n (1-\gamma)^n\;\mbox{{if, and only if,}}\; A_n = \frac{a_n}{(1-\gamma)^n}.
\end{align*}
	Clearly, for $f \in \mathcal{B}(\Omega_{\gamma})$ yields the valid sharp majorant bound:  
	\[
	\sum_{n=k}^{\infty} \frac{|a_{n}|}{(1-\gamma)^{n}} \rho^{n} + \left( \frac{\rho^{-k}}{1+\frac{|a_k|}{(1-\gamma)^k}} + \frac{\rho^{1-k}}{1-\rho} \right) \sum_{n=k}^{\infty} \frac{|a_{n}|^{2}}{(1-\gamma)^{2n}} \rho^{2n} \le 1.
	\]
	Substituting the coefficient $A_n = {a_n}/{(1-\gamma)^n}$ into this inequality
	\[
	\sum_{n=k}^{\infty} |A_n| \rho^n + \left( \frac{\rho^{-k}}{1+|A_k|} + \frac{\rho^{1-k}}{1-\rho} \right) \|f\|_\rho^2 \le 1
	\]
	where $\|f\|_\rho^2 = \sum_{n=k}^{\infty} |A_n|^2 \rho^{2n}$. This inequality holds for all $\rho \le \rho^*_k$, where $\rho^*_k \in (0,1)$ is the unique positive root of the algebraic constraint equation
	\[
	\rho^{k+1}-3\rho^{k}-2\rho+2=0.
	\]
	The sharpness of the radius $\rho^*_k$ translates directly from the optimality verified in Theorem \ref{thm-3.2} via the asymptotic behavior of the testing function:  
	\[
	f(z) = (\gamma + (1-\gamma)z)^k = (1-\gamma)^k \left( z + \frac{\gamma}{1-\gamma} \right)^k.
	\]
	As $\gamma \to 1^-$, the geometric bound is preserved, confirming that $\rho^*_k$ is the best possible radius. {This completes the proof.}
\end{proof}
\begin{proof}[\bf Proof of Theorem \ref{thm-3.2}]
{Suppose that $f \in \mathcal{B}(\Omega_\gamma)$ and $f$ is given by \eqref{Eq-33.1}.} By following the same line of argument as in the proof of Theorem \ref{thm-3.1} and applying inequality \eqref{eq-3.1}, we obtain
	\begin{align*}
		&\sum_{n=k}^{\infty}\frac{|a_n|}{(1-\gamma)^n}\rho^n+\frac{\rho^{1-k}}{1-\rho}||f||^2_{\rho}\\&\leq\rho^k\left(|b_0|+\frac{\rho}{1-\rho}(1-|b_0|^2)-\left(\frac{1}{1+|b_0|}+\frac{\rho}{1-\rho}\right)\sum_{n=k+1}^{\infty}\frac{|a_n|^2}{(1-\gamma)^{2n}}\rho^{2n-2k}\right)\nonumber\\&\quad+\frac{\rho^{1-k}}{1-\rho}\sum_{n=k}^{\infty}\frac{|a_n|^2}{(1-\gamma)^{2n}}\rho^{2n}\\&=\rho^k\left(|b_0|+\frac{\rho}{1-\rho}-\left(\frac{1}{1+|b_0|}\right)\sum_{n=k+1}^{\infty}\frac{|a_n|^2}{(1-\gamma)^{2n}}\rho^{2n-2k}\right)\\&=|b_0|\rho^k+\frac{\rho^{k+1}}{1-\rho}-\frac{1}{1+|b_0|}\sum_{n=k+1}^{\infty}\frac{|a_n|^2}{(1-\gamma)^{2n}}\rho^{2n-k}\\&=|b_0|\rho^k+\frac{\rho^{k+1}}{1-\rho}+\frac{|b_0|^2\rho^k}{1+|b_0|}-\frac{\rho^{-k}}{1+|b_0|}\sum_{n=k}^{\infty}\frac{|a_n|^2}{(1-\gamma)^{2n}}\rho^{2n}
	\end{align*}
and thus, we have 
	\begin{align}\label{eq-3.2}
		\sum_{n=k}^{\infty}\frac{|a_n|}{(1-\gamma)^{n}}\rho^n&+\left(\frac{\rho^{-k}}{1+|b_0|}+\frac{\rho^{1-k}}{1-\rho}\right)||f||_\rho^2\leq\rho^k\left(|b_0|+\frac{\rho}{1-\rho}+\frac{|b_0|^2}{1+|b_0|}\right)\\&=1+Q_{\rho}(|b_0|),\nonumber
	\end{align}
	where
	\begin{align*}
		Q_{\rho}(x)=\rho^k\left(x+\frac{\rho}{1-\rho}+\frac{x^2}{1+x}\right)-1,\;\mbox{for}\;x\in[0,1]\;\mbox{and}\;\rho\in[0,1).
	\end{align*}
	Since
	\begin{align*}
		\frac{\partial Q_{\rho}(x)}{\partial x} = \rho^k \left( \frac{1 + 4x + 2x^2}{(1+x)^2} \right)\geq0,
	\end{align*}
	 it follows that $Q(x,\rho)$ is monotonically increasing and hence, 
	 \begin{align*}
	 	Q_{\rho}(x)\leq Q_{\rho}(1)=\rho^k\left(\frac{\rho}{1-\rho}+\frac{3}{2}\right)-1:=R_k(\rho).
	 \end{align*}
	It is {easy to see} that $R_k$ is a monotonically increasing function of $\rho$. Moreover, $R_k(0)=-1$ and $\lim_{\rho\to 1}R_k(\rho)=+\infty$. Hence, the equation $R_k(\rho)=0$ possesses a unique positive root $\rho_k^*$ in the interval $(0,1)$. Consequently, $R_k(\rho)\leq 0$ for all $\rho\leq \rho_k^*$, where $\rho_k^*$ is the unique positive root in $(0,1)$ of the equation
	\begin{align*}
		\rho^{k+1}-3\rho^k-2\rho+2=0.
	\end{align*}
	 It follows from the inequality \eqref{eq-3.2} that
	 \begin{align*}
	 	\sum_{n=k}^{\infty}\frac{|a_n|}{(1-\gamma)^{n}}\rho^n+\left(\frac{\rho^{-k}}{1+|b_0|}+\frac{\rho^{1-k}}{1-\rho}\right)||f||_\rho^2\leq 1\;\;\mbox{for}\;\; \rho\leq\rho^*_k.
	 \end{align*}
	 To establish the sharpness of the radius, we consider the extremal function 
	 \begin{align*}
	 	f(z) = (\gamma + (1-\gamma)z)^k = (1-\gamma)^k (z + \frac{\gamma}{1-\gamma})^k.
	 \end{align*}
	 A {simple} computations yields that $a_k = (1-\gamma)^k$ and $a_n = 0$ for $n \geq k+1$, which implies $|b_0| = |a_k|/(1-\gamma)^k = 1$. Consequently, the majorant series and the $H^2$-norm simplify to 
	 \[
	 \sum_{n=k}^{\infty} \frac{|a_n|}{(1-\gamma)^n} \rho^n = \rho^k \quad \text{and} \quad \|f\|_\rho^2 = \rho^{2k}.
	 \]
	 Substituting these values into the left-hand side of the inequality yields
	 \begin{align*}
	 	\sum_{n=k}^{\infty}\frac{|a_n|}{(1-\gamma)^{n}}\rho^n+\left(\frac{\rho^{-k}}{1+|b_0|}+\frac{\rho^{1-k}}{1-\rho}\right)\|f\|_\rho^2 
	 	&= \rho^k + \left(\frac{\rho^{-k}}{2}+\frac{\rho^{1-k}}{1-\rho}\right)\rho^{2k} \\
	 	&= \rho^k \left(\frac{3}{2}+\frac{\rho}{1-\rho}\right) \\
	 	&= 1 + \frac{\rho^k(3-\rho)-2(1-\rho)}{2(1-\rho)},
	 \end{align*}
	 which is strictly greater than $1$ {if, and only if,} $\rho^k(3-\rho)-2(1-\rho) > 0$, or equivalently, 
	 \[
	 \rho^{k+1}-3\rho^k-2\rho+2 < 0.
	 \]
	 Since $\rho^*_k$ is the unique root of the equation $\rho^{k+1}-3\rho^k-2\rho+2 = 0$ in $(0,1)$, the above expression is strictly greater than $1$ for all $\rho > \rho^*_k$. This confirms that the radius $\rho^*_k$ is {the} best possible and completes the proof.
\end{proof}
 \begin{figure}[htbp]
	\centering
	\includegraphics[width=0.95\textwidth]{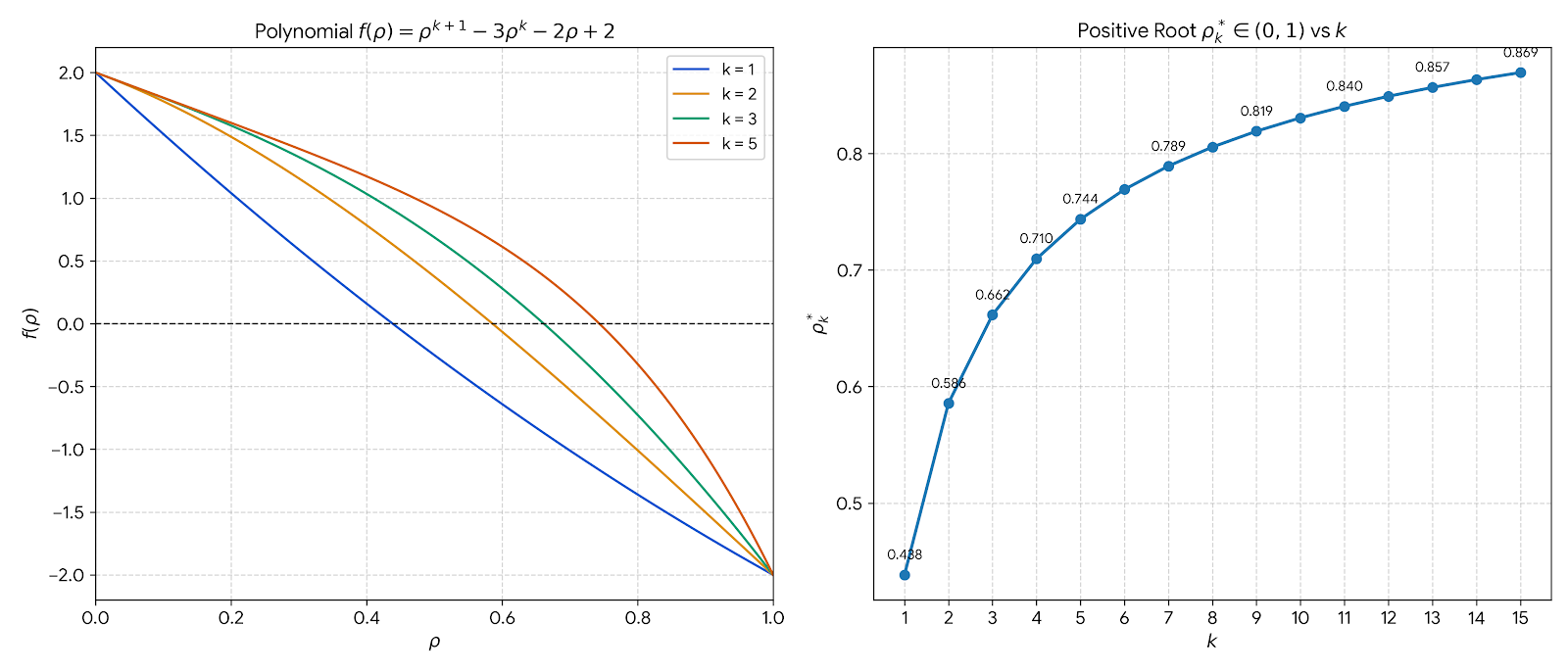}
	\caption{Analysis of the polynomial $f(\rho)$ (left) and the trajectory of its unique positive root $\rho_k^*$ across values of $k$ (right).}
	\label{fig:theorem_3_2_visuals}
\end{figure}
By utilizing the coefficient estimates established in Lemma \ref{Lem-3.1}, we provide a sharp refinement and generalization of item (c) of Theorem D by embedding the problem into the more general setting of the shifted disk $\Omega_{\gamma}$.
\begin{thm}\label{thm-3.3}
	Suppose that $f \in \mathcal{B}(\Omega_\gamma)$ and $f$ is given by \eqref{Eq-33.1}. Then 
	\begin{align*}
		\sum_{n=k}^{\infty}\frac{|a_n|}{(1-\gamma)^{n}}\rho^n+\left(\frac{r^{-k}}{1+a}+\frac{\rho^{1-k}}{1-\rho}\right)||f||_\rho^2\leq1\;\;\mbox{for}\;\;\rho\leq\rho(a),
	\end{align*}
	 where $\rho_k(a)$ is the root in $(0, 1)$ of the equation
	\begin{align}\label{Eq-33.8}
		(2a^2-1)\rho^{k+1}-a(1+2a)\rho^k-(1+a)\rho+1+a=0.
	\end{align}
	The radius $\rho_k(a)$ is {the} best possible for each $a\in(0,1)$.
\end{thm}
\begin{rem}
	By assigning $k=0$ in Theorem \ref{thm-3.3}, we deduce the following immediate corollary functions on the shifted disk $\Omega_{\gamma}$.
\end{rem}
\begin{cor}
	Suppose that $f \in \mathcal{B}(\Omega_\gamma)$ and $f$ is given by \eqref{Eq-33.1}. Then
	\begin{align*}
		\sum_{n=0}^{\infty}\frac{|a_n|}{(1-\gamma)^{n}}\rho^n+\left(\frac{1}{1+a}+\frac{\rho}{1-\rho}\right)||f||_\rho^2\leq1\;\mbox{for}\;\rho\leq\rho_0(a),
	\end{align*}
	where $\rho_0(a)$ is the unique positive root of the equation
	\begin{align*}
		(2a^2-1)\rho - a(1+2a) - (1+a)\rho + 1 + a = 0,
	\end{align*}
	which simplifies to
	\begin{align*}
		(2a^2 - a - 2)\rho + (1 - 2a^2) = 0.
	\end{align*}
	The radius $\rho_0(a) = {{(2a^2 - 1)}/{(2a^2 - a - 2)}}$ is sharp for any $a\in(0,1)$.
\end{cor}

By choosing $k = 1$ in Theorem \ref{thm-3.3}, we deduce the following corollary on shifted disk $\Omega_{\gamma}$.

\begin{cor}\label{cor-3.5}Suppose that $f \in \mathcal{B}(\Omega_\gamma)$ and $f(z) = \sum_{n=1}^{\infty} a_n \left(z+\frac{\gamma}{1-\gamma}\right)^n$. Then\begin{align*}\sum_{n=1}^{\infty}\frac{|a_n|}{(1-\gamma)^{n}}\rho^n+\left(\frac{r^{-1}}{1+a}+\frac{1}{1-\rho}\right)||f||_\rho^2\leq1\;\mbox{for}\;\rho\leq\rho_1(a),\end{align*}where $\rho_1(a)$ is the positive root of the equation\begin{align*}(2a^2-1)\rho^{2}-a(1+2a)\rho-(1+a)\rho+1+a=0,\end{align*}which simplifies to
	\begin{align*}
		(2a^2-1)\rho^{2}-(2a^2+2a+1)\rho+1+a=0.
	\end{align*}
	The radius $\rho_1(a)$ is sharp for any $a\in(0,1)$.
	\end{cor}

By setting $k=0$ and $\gamma=0$, we deduce a corollary which generalizes \cite[Theorem 3.3]{Ponnusamy-Vijayakumar-Wirths-RM-2020} (\textit{i.e.,} Theorem D) on unit disk $\mathbb{D}$.

\begin{cor}\label{cor-3.6}
	Suppose that $f \in \mathcal{B}(\mathbb{D})$ and $f(z) = \sum_{n=0}^{\infty} a_n z^n$. Then
	\begin{align*}
		\sum_{n=0}^{\infty}|a_n|\rho^n+\left(\frac{1}{1+a}+\frac{\rho}{1-\rho}\right)||f||_\rho^2\leq1\;\mbox{for}\;\rho\leq\rho_0(a),
	\end{align*}
	where \begin{align*}
		\rho_0(a) = \frac{2a^2 - 1}{2a^2 - a - 2}
	\end{align*} is the unique  root in $(0, 1)$ of the equation
	\begin{align*}
		(2a^2 - a - 2)\rho + (1 - 2a^2) = 0.
	\end{align*}
	The radius $\rho_0(a)$ is sharp for any $a\in(0,1)$.
\end{cor}

By setting $k=1$ and $\gamma=0$, we deduce a corollary which generalizes \cite[Theorem 3.3]{Ponnusamy-Vijayakumar-Wirths-RM-2020} on unit disk $\mathbb{D}$.

\begin{cor}\label{cor-3.7}
	Suppose that $f \in \mathcal{B}(\mathbb{D})$ and $f(z) = \sum_{n=1}^{\infty} a_n z^n$. Then
	\begin{align*}
		\sum_{n=1}^{\infty}|a_n|\rho^n+\left(\frac{r^{-1}}{1+a}+\frac{1}{1-\rho}\right)||f||_\rho^2\leq1\;\mbox{for}\;\rho\leq\rho_1(a),
	\end{align*}
	where 
	\begin{align*}
		\rho_1(a) = \frac{2a^2+2a+1 - \sqrt{4a^4+8a+5}}{2(2a^2-1)}
	\end{align*} is the unique root in $(0, 1)$ of the equation
	\begin{align*}
		(2a^2-1)\rho^{2}-(2a^2+2a+1)\rho+1+a=0.
	\end{align*}
	The radius $\rho_1(a)$ is sharp for any $a\in(0,1)$.
\end{cor}

\begin{proof}[\bf Proof of Theorem \ref{thm-3.3}]
	{Suppose that $f \in \mathcal{B}(\Omega_\gamma)$ and $f$ is given by \eqref{Eq-33.1}.} By following the same line of argument as in the proof of Theorem \ref{thm-3.1}, together with inequality \eqref{eq-3.2}, we obtain
\begin{align}\label{Eq-3.3}
	\sum_{n=k}^{\infty}\frac{|a_n|}{(1-\gamma)^{n}}\rho^n+\left(\frac{\rho^{-k}}{1+|b_0|}+\frac{\rho^{1-k}}{1-\rho}\right)||f||_\rho^2&\leq\rho^k\left(|b_0|+\frac{\rho}{1-\rho}+\frac{|b_0|^2}{1+|b_0|}\right)\\&=1-\frac{F_{\rho}(|b_0|)}{(1-\rho)(1+|b_0|)}\nonumber,
\end{align}
where 
\begin{align*}
	F_{\rho}(a)=(2a^2-1)\rho^{k+1}-a(1+2a)\rho^k-(1+a)\rho+1+a.
\end{align*}
It follows that
\begin{align}\label{eq-3.3}
	\frac{\partial F_{\rho}}{\partial \rho}&=(k+1)(2a^2-1)\rho^k-ak(1+2a)\rho^{k-1}-(1+a)=\rho^{k-1}H_{\rho}(a)\\&\quad-(1+a),\nonumber
\end{align}
where 
\begin{align*}
	H_{\rho}(a)=(k+1)(2a^2-1)\rho-ak(1+2a).
\end{align*}
For $a\in(0,1/\sqrt{2}]$, we readily observe that $H_{\rho}(a)<0$. Consequently, inequality \eqref{eq-3.3} implies that $\frac{\partial F_{\rho}}{\partial\rho}<0$.\vspace{1.5mm}

For $a\in(1/\sqrt{2},1)$, since $0<\rho<1$, we have
\begin{align*}
	H_{\rho}(a)\leq& (k+1)(2a^2-1)-ak(1+2a)\\=&2a^2-1-k(1+a)\\\leq&2a^2-a-2\\=&2\left(a-\frac{1}{4}\right)^2-\frac{17}{8}<0.
\end{align*}
\begin{table}[htbp]
	\centering
	\small
	\begin{tabular}{|c|c|c|c|c|}
		\hline
		$k$ & $a = 0.2$ & $a = 0.4$ & $a = 0.6$ & $a = 0.8$ \\
		\hline
		0  & 0.4340 & 0.3269 & 0.1489 & ---    \\ \hline
		1  & 0.5925 & 0.5598 & 0.5218 & 0.4806 \\ \hline
		2  & 0.6677 & 0.6497 & 0.6295 & 0.6079 \\ \hline
		3  & 0.7147 & 0.7029 & 0.6897 & 0.6758 \\ \hline
		4  & 0.7478 & 0.7392 & 0.7298 & 0.7198 \\ \hline
		5  & 0.7726 & 0.7661 & 0.7589 & 0.7513 \\ \hline
		6  & 0.7922 & 0.7870 & 0.7813 & 0.7753 \\ \hline
		7  & 0.8081 & 0.8038 & 0.7992 & 0.7942 \\ \hline
		8  & 0.8213 & 0.8177 & 0.8138 & 0.8097 \\ \hline
		9  & 0.8325 & 0.8295 & 0.8262 & 0.8227 \\ \hline
		10 & 0.8422 & 0.8396 & 0.8367 & 0.8337 \\
		\hline
	\end{tabular}\vspace{1.2mm}
	
	\caption{{The sharp radii $\rho_k(a)$ of Theorem \ref{thm-3.3} are shown in this table for various values of $k$ and $a$.}}
	\label{tab:theorem_3_3_radii}
\end{table}
Hence, inequality \eqref{eq-3.3} implies that $\frac{\partial F_{\rho}}{\partial \rho}<0$ for $a\in(1/\sqrt{2},1)$. Thus, $H_{\rho}$ is monotonically decreasing function of $\rho$.\vspace{1.2mm}

Moreover, we have $F_{0}(a)=1+a>0$ and $\lim_{\rho\to 1}F_{\rho}(a)=-(1+a)<0$. Hence, the equation $F_{\rho}(a)=0$ possesses a unique positive root $\rho(a)$ in the interval $(0,1)$. Consequently, we have $F_{\rho}(a)\geq 0$ for $\rho\leq \rho_k(a)$. It follows from the inequality \eqref{Eq-3.3} that 
\begin{align*}
	\sum_{n=k}^{\infty}\frac{|a_n|}{(1-\gamma)^{n}}\rho^n+\left(\frac{\rho^{-k}}{1+|b_0|}+\frac{\rho^{1-k}}{1-\rho}\right)||f||_\rho^2\leq 1\;\;\mbox{for}\;\; \rho\leq\rho_k(a),
\end{align*}
where $\rho_k(a)$ is the unique positive root of the equation 
\begin{align*}
(2a^2-1)\rho^{k+1}-a(1+2a)\rho^k-(1+a)\rho+1+a=0.	
\end{align*}
For the sharpness of the result, by considering the same function used in the proof of the sharpness part of Theorem \ref{thm-3.1}, we obtain
\begin{align*}
	&\sum_{n=k}^{\infty}\frac{|C_n|}{(1-\gamma)^{n}}\rho^n+\left(\frac{\rho^{-k}}{1+\frac{|C_k|}{(1-\gamma)^k}}+\frac{\rho^{1-k}}{1-\rho}\right)||f||_\rho^2\\&=a\rho^k+\frac{(1-a^2)\rho^{k+1}}{1-a\rho}+\left(\frac{\rho^{-k}}{1+a}+\frac{\rho^{1-k}}{1-\rho}\right)\left(a^2\rho^{2k}+\frac{(1-a^2)^2\rho^{2k+2}}{1-a^2\rho^2}\right)\\&=a\rho^k+\frac{(1-a^2)\rho^{k+1}}{1-\rho}+\frac{a^2\rho^k(1+a\rho)}{(1+a)(1-\rho)}\\&=\rho^k\left(a+\frac{\rho}{1-\rho}+\frac{a^2}{1+a}\right)\\&=1-\frac{F_{\rho}(a)}{(1-\rho)(1+a)},
\end{align*}
which is greater than $1$ when $\rho>\rho_k(a)$. This completes the proof.
\end{proof}
Under the limiting behavior of the parameter $\gamma$, we bridge the gap between bounded circular boundaries and vertical line boundaries. This geometric transition yields a generalized, refined Bohr-type inequality embedded with the $H^2$-norm optimization, establishing a precise radius whose sharpness is verified by analytic extremals for arbitrary zero multiplicities.
\begin{thm}\label{thm-3.3-halfplane}
	Let $\mathbb{H}_1 = \{ w \in \mathbb{C} : \operatorname{Re}(w) < 1 \}$ and assume that $f \in \mathcal{B}(\mathbb{H}_1)$ is represented by the series expansion \eqref{Eq-33.4}. Then, for any fixed parameters $a \in (0,1)$ and $r \in (0,1)$, the following inequality holds:
	\begin{align*}
		\sum_{n=k}^{\infty} |A_n| \rho^n + \left( \frac{r^{-k}}{1 + a} + \frac{\rho^{1-k}}{1 - \rho} \right) \|f\|_\rho^2 \le 1 \quad \text{for } \rho \le \rho_k(a),
	\end{align*}
	where $\|f\|_\rho^2 = \sum_{n=k}^{\infty} |A_n|^2 \rho^{2n}$, and $\rho_k(a)$ is the unique positive root in $(0,1)$ of the algebraic equation \eqref{Eq-33.8}. The radius $\rho_k(a)$ is {the} best possible for each $a \in (0,1)$.
\end{thm}
\begin{proof}[\bf Proof of Theorem \ref{thm-3.3-halfplane}]
	The right-bounded half-plane $\mathbb{H}_1 = \{ w \in \mathbb{C} : \operatorname{Re}(w) < 1 \}$ is the limiting domain of the family of nested shifted disks  $\Omega_{\gamma}$ as $\gamma \to 1^-$. Since $\Omega_{\gamma} \subsetneq \mathbb{H}_1$ for all $\gamma \in [0,1)$, any $f \in \mathcal{B}(\mathbb{H}_1)$ restricted to $\Omega_{\gamma}$ belongs to $\mathcal{B}(\Omega_{\gamma})$. Matching the Taylor expansion $f(w) = \sum_{n=k}^{\infty} A_n (w + 1)^n$ with $f(z) = \sum_{n=k}^{\infty} a_n ( z + \frac{\gamma}{1-\gamma} )^n$ establishes the identity $A_n = a_n(1-\gamma)^{-n}$.\vspace{1.2mm}
	
	Applying Theorem \ref{thm-3.3} to $f \in \mathcal{B}(\Omega_{\gamma})$, we have the sharp majorant inequality
	\[
	\sum_{n=k}^{\infty} \frac{|a_{n}|}{(1-\gamma)^{n}} \rho^{n} + \left( \frac{r^{-k}}{1+a} + \frac{\rho^{1-k}}{1-\rho} \right) \sum_{n=k}^{\infty} \frac{|a_{n}|^{2}}{(1-\gamma)^{2n}} \rho^{2n} \le 1.
	\]
	Substituting $A_n = a_n(1-\gamma)^{-n}$ directly yields
	\[
	\sum_{n=k}^{\infty} |A_n| \rho^n + \left( \frac{r^{-k}}{1+a} + \frac{\rho^{1-k}}{1-\rho} \right) \|f\|_\rho^2 \le 1,
	\]
	where $\|f\|_\rho^2 = \sum_{n=k}^{\infty} |A_n|^2 \rho^{2n}$. This holds for $\rho \le \rho_k(a)$, where $\rho_k(a) \in (0,1)$ is the unique positive root of $(2a^{2}-1)\rho^{k+1}-a(1+2a)\rho^{k}-(1+a)\rho+1+a=0$.\vspace{1.2mm}
	
	The sharpness follows by letting $\gamma \to 1^-$ for the conformal mapping 
	\begin{align*}
		f(z) = \sum_{n=k}^{\infty} C_n (z + \frac{\gamma}{1-\gamma})^n
	\end{align*} generated by the tracking parameters $C_k = a(1-\gamma)^k$ and $C_n = -(1-\gamma)^n a^{n-k-1}(1-a^2)$ for $n \ge k+1$, which strictly exceeds $1$ for any $\rho > \rho_k(a)$.
\end{proof}
\section{\bf Declarations}
\noindent\textbf{Compliance of Ethical Standards}\\

\noindent\textbf{Conflict of interest.} The authors declare that there is no conflict  of interest regarding the publication of this paper.\vspace{1.5mm}

\noindent\textbf{Data availability statement.}  Data sharing not applicable to this article as no datasets were generated or analyzed during the current study.\vspace{1.5mm}

\noindent\textbf{Funding.} No fund.

\end{document}